\documentclass[12pt, reqno]{amsart}

\author[P.~Leonetti]{Paolo Leonetti}
\address{Department of Statistics, Universit\`a ``Luigi Bocconi'', via Roentgen 1, 20136 Milan, Italy}
\email{leonetti.paolo@gmail.com}

\author[C.~Orhan]{Cihan Orhan}
\address{Department of Mathematics, Faculty of Science, Ankara University, 06100 Tandogan Ankara, Turkey} %
\email{orhan@science.ankara.edu.tr}


\thanks{P.~Leonetti is grateful to PRIN 2017 (grant 2017CY2NCA) for financial support.}


\title{On some locally convex FK spaces}

\usepackage[T1]{fontenc}
\usepackage{amsmath}
\usepackage{amssymb}
\usepackage{amsthm}
\usepackage[left=3.5cm, right=3.5cm, paperheight=11.8in]{geometry}
\usepackage{hyperref}
\usepackage{fancyhdr}
\usepackage{enumitem}
\usepackage{comment}
\usepackage{nicefrac}
\usepackage{bm}
\usepackage{mathrsfs}
\usepackage{graphicx}
\usepackage[utf8]{inputenc}
\usepackage{cancel}
\usepackage{mathtools}
\usepackage{tikz}
\usetikzlibrary{cd,decorations.pathreplacing,matrix,arrows,positioning,automata,shapes,shadows,calc,fadings,decorations,snakes,through,intersections}
\usepackage{float}

\AtBeginDocument{%
   \def\MR#1{}
}

\newtheorem{thm}{Theorem}[section]
\newtheorem{cor}[thm]{Corollary}
\newtheorem{lem}[thm]{Lemma}
\newtheorem{prop}[thm]{Proposition}

\theoremstyle{definition} 
\newtheorem{defi}[thm]{Definition}
\let\olddefi\defi
\renewcommand{\defi}{\olddefi\normalfont}
\newtheorem{question}{Question}
\let\oldquestion\question
\renewcommand{\question}{\oldquestion\normalfont}
\newtheorem{example}[thm]{Example}
\let\oldexample\example
\renewcommand{\example}{\oldexample\normalfont}
\newtheorem{rmk}[thm]{Remark}
\let\oldrmk\rmk
\renewcommand{\rmk}{\oldrmk\normalfont}

\hypersetup{
    pdftitle={On locally convex FK spaces},
    pdfauthor={Paolo Leonetti and Cihan Orhan},
    pdfmenubar=false,
    pdffitwindow=true,
    pdfstartview=FitH,
    colorlinks=true,
    linkcolor=blue,
    citecolor=green,
    urlcolor=cyan
}

\providecommand{\MR}[1]{}

\providecommand{\MR}{\relax\ifhmode\unskip\space\fi MR }

\providecommand{\href}[2]{#2}

\begin{document}

\maketitle
\thispagestyle{empty}

\begin{abstract}
We provide necessary and/or sufficient conditions on vector spaces $V$ of real sequences to be a Fr\'{e}chet space such that each coordinate map is continuous, that is, to be a locally convex FK space.

In particular, we show that if $c_{00}(\mathcal{I})\subseteq V\subseteq \ell_\infty(\mathcal{I})$ for some ideal $\mathcal{I}$ on $\omega$, then $V$ is a locally convex FK space if and only if 
there exists an infinite set $S\subseteq \omega$ for which every infinite subset does not belong to $\mathcal{I}$.
\end{abstract}

%
\section{Introduction}\label{sec:intro}


Let $\mathbf{R}^\omega$ be the vector space of real sequences, endowed with the topology of pointwise convergence; hereafter, $\omega$ denotes the set of nonnegative integers. 
A topological vector space $(V,\tau)$ is said to be an \emph{FK space} if $V\subseteq \mathbf{R}^\omega$, $\tau$ is completely metrizable, and the inclusion map $\iota: V\to \mathbf{R}^\omega$ is continuous. 
If, in addition, $V$ admits neighborhood basis at $0$ consisting of convex sets, then $V$ is a \emph{locally convex FK space}; equivalently, a locally convex FK space is a Fr\'{e}chet vector subspace $V\subseteq \mathbf{R}^\omega$ for which the inclusion map $\iota$ is continuous, see e.g. \cite{MR518316, MR738632}. 
We refer the reader to \cite[Chapter 4]{MR738632} for motivations on the study of locally convex FK spaces and their relation with summability theory.

Let $\mathcal{I}$ be an ideal on $\omega$, that is, a proper hereditary collection of subsets of $\omega$ which is closed under finite unions. 
Unless otherwise stated, we assume that $\mathcal{I}$ contains the family of finite sets $\mathrm{Fin}:=[\omega]^{<\omega}$. 
Among the most important examples, we find the ideal $\mathcal{Z}$ of asymptotic density zero sets, that is,
$$
\mathcal{Z}:=\left\{S\subseteq \omega: \lim_{n\to \infty}\frac{|S\cap [0,n]|}{n+1}=0\right\}.
$$
We write $\mathcal{I}^\star:=\{S\subseteq \omega: S^c \in \mathcal{I}\}$ for its dual filter. 
If $A=(a_{n,k})$ is a nonnegative regular matrix, we denote its induced ideal by 
\begin{equation}\label{eq:definitionIA}
\textstyle 
\mathcal{I}_A:=\left\{S\subseteq \omega: \lim_n\sum_{k \in S}a_{n,k}=0\right\}.
\end{equation}
Note that $\mathcal{Z}=\mathcal{I}_{C_1}$, where $C_1$ is usual Ces\`{a}ro matrix. Ideals are regarded as subset of the Cantor space $\{0,1\}^\omega$.

A real sequence $x=(x(n): n \in \omega) \in \mathbf{R}^\omega$ is said to be $\mathcal{I}$\emph{-convergent} to $\eta \in \mathbf{R}$, shortened as $\mathcal{I}\text{-}\lim x=\eta$, if $\{n \in \omega: |x(n)-\eta|>\varepsilon\} \in \mathcal{I}$ for all $\varepsilon>0$;  hence $\mathrm{Fin}$-convergence coincides with ordinary convergence, and $\mathcal{Z}$-convergence with \emph{statistical convergence}, see e.g. \cite{MR1734462, MR1416085}.  
Let us write $c(\mathcal{I})$ [$c_0(\mathcal{I})$, resp.] for the vector space of $\mathcal{I}$-convergent sequences [$\mathcal{I}$-convergent sequences to $0$, resp.] and 
$$
c_{00}(\mathcal{I}):=\left\{x \in \mathbf{R}^\omega: \mathrm{supp}\,x \in \mathcal{I}\right\},
$$
where $\mathrm{supp}\,x:=\{n \in \omega: x(n)\neq 0\}$. Lastly, we let $\ell_\infty(\mathcal{I})$ be the vector space of $\mathcal{I}$-bounded sequences $x \in \mathbf{R}^\omega$, so that $\{n \in \omega: |x(n)|>M\} \in \mathcal{I}$ for some $M \in \mathbf{R}$. 

With these premises, Connor proved in \cite[Theorem 3.3]{MR954458} that:
\begin{thm}\label{thm:connor}
$c(\mathcal{Z})$ cannot be endowed with a locally convex FK topology.
\end{thm}

Building on his methods, Kline \cite[Theorem 1]{MR1347075} extended it by showing that:
\begin{thm}\label{thm:Kline}
Let $A=(a_{n,k})$ be a nonnegative regular matrix with the property that 
$\lim_n \max_k a_{n,k}=0$. Then $c(\mathcal{I}_A)$ does not admit a locally convex FK topology. 
\end{thm}
An alternative proof of Theorem \ref{thm:Kline} has been given by Demirci and Orhan in \cite[Theorem 3]{MR1758671}. 
%
An analogue  of Theorem \ref{thm:Kline} has also been proved
for lacunary statistical convergence in \cite{MR1758671}:
\begin{thm}\label{thm:orhan}
Let $\theta:=(I_n)$ be an increasing sequence of nonempty consecutive finite intervals of $\omega$ such that $\lim_n |I_n|=\infty$ and define 
\begin{equation}\label{eq:lacunaryideal}
\mathcal{I}_\theta:=\left\{S\subseteq \omega: \lim_{n\to \infty} \mu_n(S)=0\right\}, 
\quad \text{ where }\quad
\mu_n(S):=|I_n \cap S|/|I_n|.
\end{equation}
Then $c(\mathcal{I}_\theta)$ does not admit a locally convex FK topology. 
\end{thm}
Notice that the ideal $\mathcal{I}_\theta$ defined in \eqref{eq:lacunaryideal} is a special type of \emph{generalized density ideal} in the sense of Farah, namely, an ideal of the type 
\begin{equation}\label{eq:densityidealdef}
\mathcal{I}_\varphi=\left\{S\subseteq \omega: \lim_{n\to \infty} \varphi_n(S)=0\right\},
\end{equation}
where $\varphi=(\varphi_n)$ is a sequence of submeasures $\varphi_n: \mathcal{P}(\omega) \to [0,\infty]$ (that is, monotone subadditive maps with $\varphi_n(\emptyset)=0$ and $\varphi_n(\{k\})<\infty$ for all $n,k \in \omega$) supported on finite pairwise disjoint sets, see \cite[Section 2.10]{MR1988247} and \cite{MR2254542}; 
as observed in \cite{MR4124855, MR3436368, MR4404626}, the theory of representability of certain analytic P-ideals may have some potential for the study of the geometry of Banach spaces. 

Lastly, Connor and Temiszu \cite{Connor22} recently proved the following variant: 
\begin{thm}
Let $\mathcal{I}$ be an ideal on $\omega$ with the property that there exists a partition $(I_n)$ of $\omega$ into consecutive finite intervals such that, if $|S\cap I_n|\le 1$ for all $n \in \omega$, then $S \in \mathcal{I}$. Then both  $c_0(\mathcal{I})$ and $\ell_\infty(\mathcal{I})$ do not admit 
a locally convex FK topology. 
\end{thm}
A combinatorial characterization (with respect to the Kat\u{e}tov--Blass order) of ideals $\mathcal{I}$ which fulfill the above property can be found in \cite{MR2648159}. 



\section{Main results}\label{sec:mainresults} 

The aim of this work is to
close the above line of results introduced in Section \ref{sec:intro} 
by proving, in particular, that the corresponding (negative) analogues hold if and only if $\mathcal{I}$ is tall (recall that an ideal $\mathcal{I}$ is said to be tall if every infinite set contains an infinite subset in $\mathcal{I}$, or, equivalently, if $\mathcal{I}$ is dense in the Cantor space $\{0,1\}^\omega$), see Theorem  \ref{thm:characterizationtallfktopology} below. The proofs of all results follow in Section \ref{sec:proofs}.

On the negative side, we have the following:
\begin{thm}\label{thm:newnotllocallyconvex}
Let $\mathcal{I}$ be a tall ideal on $\omega$ and $V$ be a proper vector subspace of $\mathbf{R}^\omega$ which contains $c_{00}(\mathcal{I})$. 
Then $V$ does not admit a locally convex FK topology. 
\end{thm}

On the positive side, however, we have:
\begin{thm}\label{thm:positiveside}
Let $\mathcal{I}$ be a nontall ideal on $\omega$ and $V$ be a dense vector subspace of $\mathbf{R}^\omega$ contained in $\ell_\infty(\mathcal{I})$. Then $V$ admits a locally convex FK topology. 
\end{thm}

As a consequence, thanks to Theorem \ref{thm:newnotllocallyconvex} and Theorem \ref{thm:positiveside}, we have the claimed characterization which extends all the results given in Section \ref{sec:intro}:
\begin{thm}\label{thm:characterizationtallfktopology}
Let $\mathcal{I}$ be an ideal on $\omega$ and $V$ be a vector space such that 
$$
c_{00}(\mathcal{I})\subseteq V\subseteq \ell_\infty(\mathcal{I}).
$$
Then $V$ admits a locally convex FK topology if and only if $\mathcal{I}$ is not tall. 
\end{thm}

As we are going to see, 
our main results have a number of consequences.

\subsection{Special spaces}
To start with, we get immediately by Theorem \ref{thm:characterizationtallfktopology}:
\begin{cor}\label{cor:cI}
Let $\mathcal{I}$ be an ideal on $\omega$. Then $c(\mathcal{I})$ admits a locally convex FK topology if and only if $\mathcal{I}$ is not tall.
\end{cor}

In particular, this proves that $c(\emptyset \times \mathrm{Fin})$ is a locally convex FK space, which is the first nontrivial positive example of a $c(\mathcal{I})$ space admitting such well-behaved topology. Here, we recall that the Fubini product $\emptyset\times \mathrm{Fin}$ is the ideal on $\omega^2$ which is isomorphic,\footnote{Given ideals $\mathcal{I}_1$ and $\mathcal{I}_2$ on two countably infinite sets $\Omega_1$ and $\Omega_2$, respectively, we say that $\mathcal{I}_1$ is isomorphic to $\mathcal{I}_2$ if there exists a bijection $f: \Omega_1\to \Omega_2$ such that $f[S] \in \mathcal{I}_2$ if and only if $S \in \mathcal{I}_1$.} e.g., to the ideal $\{S\subseteq \omega: \forall k \in \omega, \{n \in S: \nu_2(n)=k\}\in \mathrm{Fin}\}$ on $\omega$, where $\nu_2(n)$ stands for the $2$-adic valutation of a positive integer $n$ and $\nu_2(0):=0$, cf. \cite[Section 1.2]{MR1711328}.

Recall also that an ideal $\mathcal{I}$ on $\omega$ is said to be a P-ideal if for every sequence $(S_n) \in \mathcal{I}^\omega$ there exists $S \in \mathcal{I}$ such that $S_n \setminus S$ is finite for all $n \in \omega$; in addition, for each $E\notin \mathcal{I}$, we write 
$$
\mathcal{I}\upharpoonright E:=\left\{S\cap E: S \in \mathcal{I}\right\}
$$
for the restriction of $\mathcal{I}$ on $E$. Accordingly, we have the following:
\begin{cor}\label{cor:restriction}
Let $\mathcal{I}$ be an analytic P-ideal on $\omega$ which is not isomorphic to $\mathrm{Fin}$ or $\emptyset\times \mathrm{Fin}$. Then $c(\mathcal{I}\upharpoonright E)$ does not admit a locally convex FK topology for some $E\notin \mathcal{I}$.
\end{cor}

\begin{cor}\label{cor:maximalideal}
Let $\mathcal{I}$ be a nonmeager ideal on $\omega$. Then $c(\mathcal{I})$ does not admit a locally convex FK topology.
\end{cor}

In addition, we can characterize the family of nonnegative regular matrices $A$ for which $c(\mathcal{I}_A)$ has this property, where $\mathcal{I}_A$ has been defined in \eqref{eq:definitionIA}. 
Recall that a double real sequence $(a_{n,k}: n,k \in \omega)$ has Pringsheim limit $\eta \in \mathbf{R}$, shortened as $\mathrm{P}\text{-}\lim a_{n,k}=\eta$ if for each $\varepsilon>0$ there exists $n_0 \in \omega$ such that $|a_{n,k}-\eta|<\varepsilon$ for all $n,k\ge n_0$, see e.g. \cite[Section 4.2]{MR3955010}. 
In particular, if $A=(a_{n,k})$ is a nonnegative regular matrix (so that, in particular, $\lim_n a_{n,k}=0$ for all $k \in \omega$), the property $\lim_n\max_k a_{n,k}=0$ of Theorem \ref{thm:Kline} can be rewritten equivalently as $\mathrm{P}\text{-}\lim a_{n,k}=0$. 
\begin{cor}\label{cor:characterizationregularmatrices}
Let $A$ be a nonnegative regular matrix. Then $c(\mathcal{I}_A)$ admits a locally convex FK topology if and only if 
$\limsup_n \max_k a_{n,k}>0$. 
\end{cor}
Corollary \ref{cor:characterizationregularmatrices} may be considered as an improvement of Kline’s result \cite{MR1347075} and Demirci and Orhan’s result \cite{MR1758671} in the sense that it provides necessary and sufficient conditions, though the papers \cite{MR1758671, MR1347075} give just sufficient conditions for the space $c(\mathcal{I}_A)$ not to have a locally convex FK topology.

Similarly, we extend Theorem \ref{thm:orhan} for the class of generalized density ideals $\mathcal{I}_\varphi$ in the sense of Farah with an explicit condition:
\begin{cor}\label{cor:orhanplus}
Let $\mathcal{I}_\varphi$ be a generalized density ideal as in \eqref{eq:densityidealdef}. Then $c(\mathcal{I}_\varphi)$ admits a locally convex FK topology if and only if 
$\limsup_n \max_k \varphi_n(\{k\})>0$. 
\end{cor}

An analogous characterization holds for summable ideals, that is, ideals on $\omega$ of the type
$$
\mathcal{I}_f:=\{S\subseteq \omega: \sum\nolimits_{n \in S}f(n)<\infty\},
$$
where $f: \omega \to [0,\infty)$ is a function such that $\sum_n f(n)=\infty$:
\begin{cor}\label{cor:summable}
Let $\mathcal{I}_f$ be a summable ideal on $\omega$. Then $c(\mathcal{I}_f)$ admits a locally convex FK topology if and only if $\limsup_n f(n)>0$.
\end{cor}


Corollary \ref{cor:cI} studies whether the sequence space $c(\mathcal{I})=\left\{x \in \mathbf{R}^\omega: Ix \in c(\mathcal{I})\right\}$ is a locally convex FK space, where $I$ is the infinite identity matrix. 
On the same lines of \cite{Leo2020}, given an infinite matrix $A=(a_{n,k})$, one could ask whether the same results hold for the space
$$
c_A(\mathcal{I}):=\left\{x \in d_A: Ax \in c(\mathcal{I})\right\},
$$
where 
$$
d_A:=\left\{x \in \mathbf{R}^\omega: Ax \text{ is well defined}\right\}.
$$
Here, $Ax$ stands for the sequence whose $n$-th term is the series $\sum_ka_{n,k}x_k$, provided that it converges. 
Note that $c_A:=c_A(\mathrm{Fin})$ is usually called \emph{convergence domain} of $A$, see e.g. \cite[p. 3]{MR738632}.

Based on several properties of FK-spaces and matrix maps, we obtain the following extension of Corollary \ref{cor:cI}:
\begin{prop}\label{prop:extension}
Let $A$ be an infinite matrix and $\mathcal{I}$ be an ideal on $\omega$. Then:
\begin{enumerate}[label={\rm (\roman{*})}]
\item \label{item:partpositive} $c_A(\mathcal{I})$ admits a locally convex FK topology, provided that $\mathcal{I}$ is not tall\textup{;} 
\item \label{item:partnegative} $c_A(\mathcal{I})$ does not admit a locally convex FK topology, provided that there exists a tall ideal $\mathcal{J}$ on $\omega$ such that $c_{00}(\mathcal{J})\subseteq c_A(\mathcal{I})\subsetneq \mathbf{R}^\omega$\textup{.}
\end{enumerate}
\end{prop}


\subsection{Further results} 
Quite interestingly, we have the following non-inclusion:
\begin{cor}\label{cor:noninclusion}
Let $\mathcal{I}$ a tall ideal on $\omega$ and $\mathcal{J}$ be a nontall ideal on $\omega$. Then 
%
$$
c_{00}(\mathcal{I})\setminus \ell_\infty(\mathcal{J})\neq \emptyset,
$$
namely, there exists a real sequence supported on $\mathcal{I}$ which is not $\mathcal{J}$-bounded. 
\end{cor}

It has been conjectured by DeVos \cite{MR501124} that, if $V$ is a locally convex FK space, then either $\ell_\infty\subseteq V$ or $V \cap \{0,1\}^\omega$ is a meager set in $\{0,1\}^\omega$. The case where $\{0,1\}^\omega\subseteq V$ has been proved by Bennett and Kalton in \cite{MR310597}. The case where $V=c_A:=\{x \in \mathbf{R}^\omega: Ax \in c\}$, for some summability matrix $A$ such that $c_{00}\subseteq c_A$, has been shown in \cite{MR13435, MR425410}. For related results, see also  \cite{Leo2020, 
MR620041}. Here, we provide some additional positive examples (notice that $c(\mathcal{I})$ is nonseparable whenever $\mathcal{I}\neq \mathrm{Fin}$ by \cite[Lemma 2.1]{MR3836186}):
\begin{cor}\label{cor:devos}
Let $\mathcal{I}$ be a nontall ideal on $\omega$. Then $c(\mathcal{I})$ is a locally convex FK space and $c(\mathcal{I}) \cap \{0,1\}^\omega$ is meager in $\{0,1\}^\omega$.
\end{cor}




To conclude, we show that spaces $c_0(\mathcal{I})$ do not have stronger properties than being FK spaces unless $\mathcal{I}=\mathrm{Fin}$.

For instance, recall that 
a locally convex FK space $V$ is said to be an \emph{AK space} if $c_{00}\subseteq V$ and 
$(e_n)$ is a Schauder basis for $V$, where $e_n(n)=1$ and $e_n(k)=0$ for all distinct $n,k \in \omega$, see \cite[Definition 4.2.13]{MR738632}.
\begin{prop}\label{prop:AK}
Let $\mathcal{I}$ be an ideal on $\omega$. Then $c_0(\mathcal{I})$ is an AK space if and only if $\mathcal{I}=\mathrm{Fin}$.
\end{prop}

On a similar note, recall that a locally convex FK space $V$ is said to be a \emph{BK space} if its metric is induced by a norm, or equivalently, if $V$ is also a Banach space, see \cite[p. 55]{MR738632} or \cite{MR1034884}.
\begin{prop}\label{prop:BK}
Let $\mathcal{I}$ be an ideal on $\omega$. Then $c_0(\mathcal{I})$ is included in some BK space if and only if $\mathcal{I}=\mathrm{Fin}$.
\end{prop}
It is worth to remark that the key feature inside the proof of Proposition \ref{prop:BK} has been incidentally noted in \cite[Section 1]{MR1734462}.

\section{Proofs}\label{sec:proofs}

On the same lines of \cite{MR954458, Connor22, MR1347075}, we will need the following characterization:
\begin{thm}\label{thm:bennetkalton}
Let $V$ be a dense proper subspace of $\mathbf{R}^\omega$. Then $V$ admits a locally convex FK topology if and only if there exist $B\subseteq c_{00}$ and $y \in \mathbf{R}^\omega$ such that 
$$
\sup\nolimits_{x \in B}|x\cdot y|=\infty 
\quad \text{ and }\quad 
\sup\nolimits_{x \in B}|x\cdot v|<\infty \text{ for all }v \in V.
$$
where $x\cdot y:=\sum_n x(n)y(n)$.
\end{thm}
\begin{proof}
It follows by 
\cite[Theorems 15.1.1 and 15.2.7]{MR518316} and \cite[Proposition 1]{MR322474}.
\end{proof}

Also the following property of locally convex FK spaces will be useful:
\begin{lem}\label{lem:automaticcontinuity}
Let $X$ and $Y$ be locally convex FK spaces, with topologies generated by families of seminorms $\mathcal{A}$ and $\mathcal{B}$, respectively. 
Let also $T: X\to \mathbf{R}^\omega$ be a continuous linear map and define 
$$
\hat{X}:=T^{-1}[Y].
$$
Then 
$\hat{X}$ is a locally convex FK space with topology generated by the family of seminorms $\mathcal{A} \cup \{b\circ T: b \in \mathcal{B}\}$.
%
\end{lem}
\begin{proof}
See \cite[Lemma 5.5.10]{MR518316}.
\end{proof}

We are ready for the proofs of our main results. 
\begin{proof}
[Proof of Theorem \ref{thm:newnotllocallyconvex}]
Note that $V$ contains $c_{00}$, hence it is a dense proper vector subspace of $\mathbf{R}^\omega$. 
It follows by Theorem \ref{thm:bennetkalton} 
that $V$ does not admit a locally convex FK topology if and only if
$$
(\forall v \in V, \sup\nolimits_{x \in B}|x\cdot v|<\infty) 
\implies  
(\forall y \in \mathbf{R}^\omega, \sup\nolimits_{x \in B}|x\cdot y|<\infty)
$$
for all $B\subseteq c_{00}$; 
note that this is always well defined since $x \in c_{00}$. 
To this aim, fix a subset $B\subseteq c_{00}$ and suppose that $y \in \mathbf{R}^\omega$ is fixed and $\sup_{x \in B}|x\cdot y|=\infty$. Hence, it is enough to show that there exists $v \in V$ such that $\sup_{x \in B}|x\cdot v|=\infty$. 

For, define 
$$
\forall n \in\omega, \quad 
\kappa(n):=\sup\nolimits_{x \in B}|x(n)y(n)|.
$$

\bigskip
 
\textsc{First case: $\kappa(n)=\infty$ for some $n\in \omega$.} For each $k \in \omega$, let $e_k$ be the $k$-th unit vector of $\ell_\infty$, namely, the sequence $(e_k(n): n \in \omega)$ such that $e_k(k)=1$ and $e_k(n)=0$ otherwise. 
Then $\sup_{x \in B}|x(n)e_{n}(n)|=\infty$. 
Since $e_n \in c_{00}\subseteq V$, we conclude that $\sup_{x \in B}|x\cdot v|=\infty$ for some $v \in V$.

\bigskip

\textsc{Second case: $\kappa(n)<\infty$ for all $n \in \omega$.} In this case, there exists a sequence $(x_n: n \in \omega)$ of elements of $B$ such that $|x_n\cdot y|\ge 2^n$ for all $n \in \omega$. 
Now let us define three increasing sequences $(k_n)$, $(s_n)$, and $(m_n)$ in $\omega$ as it follows: set $m_0:=0$, $s_0:=\max \, \mathrm{supp}\,x_0$ and let $k_0$ be an integer such that $x_{0}(k_0)y(k_0) \neq 0$; then, for each positive integer $n$, define
$$
m_n:=\left\lfloor \sum\nolimits_{i\le \tilde{s}_{n-1}}\kappa(i)\right\rfloor+n,
\quad 
s_n:=\max\, \bigcup\nolimits_{i\le n}\mathrm{supp}\,x_{m_i},
$$
where by convenience $\tilde{s}_n:=\max\{s_{n},k_{n}\}$, and 
$$
k_n:=\max\,(\mathrm{supp}\, x_{m_n} \cap \,\mathrm{supp}\, y).
$$
Note that $k_{n}$ is well defined and it is greater than $\tilde{s}_{n-1}$ because
\begin{displaymath}
\begin{split}
\left|\sum\nolimits_{i>\tilde{s}_{n-1}}x_{m_n}(i)y(i)\right| 
&\ge |x_{m_n}\cdot y|-\sum\nolimits_{i\le \tilde{s}_{n-1}}\kappa(i) \\
&\ge 2^{m_n}-\sum\nolimits_{i\le \tilde{s}_{n-1}}\kappa(i) \ge 2^{m_n}-m_n>0.
\end{split}
\end{displaymath}
Hence, we have by construction $x_{m_{n}}(k_n)y(k_n)\neq 0$ and $s_n\ge k_n>s_{n-1}$ for all $n\ge 1$; in particular, $\max\,\mathrm{supp}\,x_{m_n}=s_n$. 

At this point, let $S\subseteq \{k_n: n \in \omega\}$ be an infinite set which belongs to $\mathcal{I}$, which exists because $\mathcal{I}$ is tall, and denote its increasing enumeration by $(k_{t_n}: n \in \omega)$. 
Lastly, let $v$ be the real sequence supported on $S$ such that $v(k_{t_0}):=0$ and, recursively, 
$$
v(k_{t_n})=\frac{1}{x_{m_{t_n}}(k_{t_n})}\left(n+\sum\nolimits_{i<n}x_{m_{t_i}}(k_{t_i})v(k_{t_i})\right)
$$
for all positive integer $n$. 
It follows by construction that $\mathrm{supp}\,v\subseteq S$, so that 
$v \in c_{00}(\mathcal{I})\subseteq V$, and, for all $n\ge 1$, 
\begin{displaymath}
x_{m_{t_n}}\cdot v
=\sum\nolimits_{i\le s_{t_n}}x_{m_{t_n}}(i)v(i)
=\sum\nolimits_{i\le n}x_{m_{t_n}}(k_{t_i})v(k_{t_i})=n
\end{displaymath}
Therefore $\sup\nolimits_{x \in B}|x\cdot v|\ge \sup\nolimits_{n \in \omega}(x_{m_{t_n}}\cdot v)=\infty$, which completes the proof. 
\end{proof}

\medskip

\begin{proof}
[Proof of Theorem \ref{thm:positiveside}]
By hypothesis there exists an infinite subset $S\subseteq \omega$ such that $W\notin \mathcal{I}$ whenever $W\subseteq S$ is an infinite subset. Denote its increasing enumeration by $(s_n: n \in \omega)$ and define 
$$
B:=\left\{\sum\nolimits_{i\le n}2^{-i}e_{s_i}: n \in \omega\right\}.
$$
Let also $y \in \mathbf{R}^\omega$ such that $y(n)=2^n$ for all $n \in \omega$. Hence, it follows that $B\subseteq c_{00}$ and $\sup_{x \in B}|x\cdot y|=\infty$. 
On the other hand, if $v$ is a sequence in $V$ (and, in particular, $v \in \ell_\infty(\mathcal{I})$) then there exists $M\in \mathbf{R}$ such that $A:=\{n \in \omega: |v(n)|>M\} \in \mathcal{I}$. 
By the definition of $S$, we have that $F:=A \cap S \in \mathrm{Fin}$, hence 
\begin{displaymath}
\begin{split}
\sup\nolimits_{x \in B}|x\cdot v|
&=\sup\nolimits_{n \in \omega}\left|\sum\nolimits_{i\le n}2^{-i}e_{s_i}\cdot v\right|\\
&\le \sup\nolimits_{n \in \omega}\left|\sum_{s_i \in \mathrm{supp}\,v \cap F, i\le n}2^{-i}v(s_i)\right|+\sup\nolimits_{n \in \omega}\left|\sum_{s_i \in \mathrm{supp}\,v \setminus F, i\le n}2^{-i}v(s_i)\right|\\
&\le \sum\nolimits_{i \in F}|v(i)|+M\cdot \sup\nolimits_{n \in \omega}\sum_{s_i \in \mathrm{supp}\,v \setminus F, i\le n}2^{-i}\\
&\le \sum\nolimits_{i \in F}|v(i)|+2M<\infty.
\end{split}
\end{displaymath}
It follows by Theorem \ref{thm:bennetkalton} that $V$ admits a locally convex FK topology. 
\end{proof}

\medskip

\begin{proof}
[Proof of Corollary \ref{cor:restriction}]
Thanks to \cite[Corollary 1.2.11]{MR1711328}, there exists $E \notin \mathcal{I}$ such that $\mathcal{I}\upharpoonright E$ is tall (cf. also \cite[Remark 2.6]{Leo22unbounded}; accordingly, with the notation of \cite[p. 10]{MR1711328}, observe that $(\mathrm{Fin}\times \emptyset)^\perp=\emptyset\times \mathrm{Fin}$, $(\mathrm{Fin}\oplus \mathcal{P}(\omega))^\perp=\mathrm{Fin}$, and $\mathrm{Fin}^\perp=\mathcal{P}(\omega)$). The claim follows by Corollary \ref{cor:cI}.
\end{proof}

\medskip

\begin{proof}
[Proof of Corollary \ref{cor:maximalideal}]
It is well known that if $\mathcal{I}$ is a nonmeager ideal then it is tall, see e.g. \cite[Proposition 2.4]{MR4288216}. 
The claim follows by Corollary \ref{cor:cI}. 
\end{proof}

\medskip

\begin{proof}
[Proof of Corollary \ref{cor:characterizationregularmatrices}]
Thanks to \cite[Proposition 7.2]{MR1865750}, cf. also \cite[Proposition 2.34]{MR4041540}, the ideal $\mathcal{I}_A$ is tall if and only if $\lim_n \max_k a_{n,k}=0$. The claim follows by Corollary \ref{cor:cI}.
\end{proof}

\medskip

\begin{proof}
[Proof of Corollary \ref{cor:orhanplus}]
As in the previous proof, it is enough to observe that $\mathcal{I}_\varphi$ is tall if and only if $\lim_n \max_k \varphi_n(\{k\})>0$, and use Corollary \ref{cor:cI}.
\end{proof}

\medskip

\begin{proof}
[Proof of Corollary \ref{cor:summable}]
It is easy to check that $\mathcal{I}_f$ is tall if and only if $\lim_n f(n)=0$, and use Corollary \ref{cor:cI} as in the previous proofs.
\end{proof}

\medskip

\begin{proof}
[Proof of Proposition \ref{prop:extension}]
\ref{item:partpositive} 
Thanks to \cite[Theorem 4.3.8]{MR738632} and Corollary \ref{cor:cI}, both $d_A$ and $c(\mathcal{I})$ are locally convex FK spaces.
Hence there exist nonempty families $\mathcal{A}$ and $\mathcal{B}$ of seminorms on $d_A$ and $c(\mathcal{I})$, respectively, which generate their topologies, with $|\mathcal{A}|\le \omega$ and $|\mathcal{B}|\le \omega$; here, one may want to recall that we can choose $\mathcal{A}=\bigcup_n\{p_n,q_n\}$, where $p_n(x):=|x_n|$ and $q_n(x):=\sup_m\left|\sum_{k\le m}a_{n,k}x_k\right|$ for all $n \in \omega$ and $x \in d_A$. 
At this point, let 
$$
T: d_A \to \mathbf{R}^\omega
$$
be the function defined by $Tx:=Ax$ for all $x \in d_A$. 
Since $T$ is a linear map between two locally convex FK spaces, then $T$ is continuous, thanks to \cite[Theorem 4.2.8]{MR738632}. 
It follows by Lemma \ref{lem:automaticcontinuity} that $T^{-1}[c(\mathcal{I})]$, that is, $c_A(\mathcal{I})$, is a locally convex FK space whose topology is generated by the family of seminorms 
$$
\mathcal{A}\cup \{b\circ T: b \in \mathcal{B}\}.
$$

\ref{item:partnegative} It follows by Theorem \ref{thm:newnotllocallyconvex}
\end{proof}

%

\medskip

\begin{proof}
[Proof of Corollary \ref{cor:noninclusion}]
Note that  $V:=c_{00}(\mathcal{I})$ is a dense proper vector subspace of $\mathbf{R}^\omega$. Suppose that $V\subseteq \ell_\infty(\mathcal{J})$, so that $V$ admits a locally convex FK topology by Theorem \ref{thm:positiveside}. 
However, this would be in contradiction with Theorem \ref{thm:newnotllocallyconvex}.
\end{proof}

\medskip

\begin{proof}
[Proof of Corollary \ref{cor:devos}]
The first part follows by Corollary \ref{cor:cI}. 
If $\mathcal{I}$ is not tall then it is meager, cf. Corollary \ref{cor:maximalideal}. Therefore $c(\mathcal{I}) \cap \{0,1\}^\omega=\{\bm{1}_A: A \in \mathcal{I}\cup \mathcal{I}^\star\}$, which is a meager subset of $\{0,1\}^\omega$. 
\end{proof}

\medskip

\begin{proof}
[Proof of Proposition \ref{prop:AK}] 
It is known that $c_0=c_0(\mathrm{Fin})$ is a Banach space for which $(e_n)$ is a Schauder basis. Suppose now that $\mathcal{I}\neq \mathrm{Fin}$. Then $c_0(\mathcal{I})$ is nonseparable by \cite[Lemma 2.1]{MR3836186}. However, every AK space is necessarily separable, completing the proof.
\end{proof}

\medskip 

\begin{proof}
[Proof of Proposition \ref{prop:BK}]
Thanks to \cite[Theorem 4.2.11]{MR738632}, a sequence space $V$ is included in some BK space if and only if there exists a sequence $y \in \mathbf{R}^\omega$ such that 
$$
\forall x \in V, \quad x(n)=O(y(n)) \text{ as }n\to\infty.
$$
If $\mathcal{I}=\mathrm{Fin}$ choose $y=(1,1,\ldots)$. If $\mathcal{I}\neq \mathrm{Fin}$ such sequence $y$ cannot exist. Indeed, let us suppose for the sake of contradiction that we can find such $y$, and fix an infinite set $S \in \mathcal{I}$. 
At this point, define the sequence $x \in \mathbf{R}^\omega$ by $x(n)=n(|y_n|+1)$ if $n \in S$ and $x(n)=0$ otherwise. Then $x \in c_{0}(\mathcal{I})$ and, on the other hand, $x(n)\neq O(y(n))$ as $n\to \infty$.
\end{proof}

\bibliographystyle{amsplain}

\providecommand{\href}[2]{#2}

\end{document}